\documentclass[a4paper,12pt]{article}

\usepackage{amsmath,amsthm,amssymb,enumerate}
\usepackage{xcolor}
\usepackage{float}
\usepackage{multirow}
\usepackage{longtable}
\usepackage{diagbox}
\usepackage{soul}
\usepackage{cancel}
\usepackage{mathrsfs}
\usepackage[affil-it]{authblk}
\usepackage{gensymb}
\usepackage{enumitem}
\usepackage[square,sort&compress,comma,numbers]{natbib}
\usepackage{bm}
\usepackage{hyperref}
\usepackage{subfig}
\usepackage{graphicx}
\usepackage[margin=2cm]{geometry}
\usepackage{tikz}
\usepackage{esint}
\usepackage{caption}
\usetikzlibrary{shapes,calc}
\usepackage{verbatim}
\usepackage{array}
\usepackage{bm}

\usepackage{newtxtext,newtxmath}

\usepackage{multirow}
\usepackage{marginnote}
\chardef\bslash=`\\ 

\usepackage{mathtools}  
\mathtoolsset{showonlyrefs} 
\usepackage[normalem]{ulem}
\DeclareMathAlphabet\mathbfcal{OMS}{cmsy}{b}{n}

\newcommand{\mathbi}[1]{{\boldsymbol #1}}
\def\N{\mathbb{N}}

\def\R{\mathbb{R}}

\def\err{\mathrm{err}}

\def\dist{{\rm dist}}

\def\Poly{\mathbb{P}}

\def\P{\mathcal{P}}

\def\Proj{\P_{\!\mesh}}
\def\<{\langle}
\def\>{\rangle}

\def\dsp{\displaystyle}

\def\norm#1#2{\Vert#1\Vert_{#2}}

\newcounter{cst}

\def\bt{\begin{theorem}}
	\def\et{\end{theorem}}
\def\bl{\begin{lemma}}
	\def\el{\end{lemma}}
\def\bc{\begin{corollary}}
	\def\ec{\end{corollary}}
\def\bd{\begin{definition}}
	\def\ed{\end{definition}}
\def\br{\begin{remark}}
	\def\er{\end{remark}}

\newcommand{\disc}{{\mathcal D}}

\def\ad{{\rm ad}}
\def\Prh{{\rm Pr}_h}
\def\Uad{\mathcal U_\ad}
\def\Uadh{\mathcal U_{\ad,h}}
\def\Uh{\mathcal U_{h}}

\def \hessian{\mathcal H}
\def \hb{\hessian^{B}}

\def \hbd{\hessian_\disc^{B}}
\def \wdspace{H^B}
\def\cv{K}
\def \symd{\mathcal S_d}
\def \sym2{\mathcal S_2}
\def\cell{K}

\newcommand{\polyd}{{\mathcal T}}

\newcommand{\mesh}{{\mathcal M}}

\newcommand{\edge}{{\sigma}}

\newcommand{\edges}{{\mathcal F}}              
\newcommand{\edgescv}{{{\edges}_\cv}}  
  
\newcommand{\edgesext}{{{\edges}_{\rm ext}}} 
\newcommand{\edgesint}{{{\edges}_{\rm int}}}

\newcommand{\centers}{\mathcal{P}}
\newcommand{\x}{\mathbi{x}}

\newcommand{\centeredge}{\overline{\mathbi{x}}_\edge} 

\def\minmod{\mathop{\rm minmod}}

\newcommand{\bu}{\overline{u}}

\def\stab{\mathfrak{S}}
\newcommand{\fl}{{\rm for\,all\,}}
\newcommand{\be}{\begin{equation}}
\newcommand{\ee}{\end{equation}}

\renewcommand{\O}{\Omega}

\def\dr{\partial}

\renewcommand{\d}{{\:\rm d}}

\newcommand{\ba}{\begin{array}{llll}   }
	\newcommand{\bac}{\begin{array}{c}}
		\newcommand{\bari}{\begin{array}{r}}
			\newcommand{\ea}{\end{array}}

		\newcommand{\NORM}[1]{{\left\vert\kern-0.25ex\left\vert\kern-0.25ex\left\vert #1 
				\right\vert\kern-0.25ex\right\vert\kern-0.25ex\right\vert}}

		\newcommand{\by}{\overline y}
		\newcommand{\bp}{\overline  p}
		\newcommand{\tu}{\widetilde{u}}
		\newcommand{\ud}{\overline{u}_d}
		\newcommand{\yd}{\overline{y}_d}
		
		\newcommand{\dm}{\disc_m}

		\def\assum#1{{\rm\textbf{(A#1)}}}
		
		\newtheorem{theorem}{Theorem}[section]
		\newtheorem{remark}[theorem]{Remark}
		
		\newtheorem{lemma}[theorem]{Lemma} 
		\newtheorem{definition}[theorem]{Definition}
		\newtheorem{proposition}[theorem]{Proposition}
		\newtheorem{corollary}[theorem]{Corollary}

		\numberwithin{equation}{section}
		\def\WS{{\rm WS}}

		\def\XXint#1#2#3{{\setbox0=\hbox{$#1{#2#3}{\int}$ }
				\vcenter{\hbox{$#2#3$ }}\kern-.6\wd0}}

		\usepackage[normalem]{ulem}
		\definecolor{violet}{rgb}{0.580,0.,0.827}
		
		\usepackage{color,colordvi}

		\newcounter{cexp}
		\catcode`\@=11
		\def\terml#1{T_{\refstepcounter{cexp}\@bsphack
				\protected@write\@auxout{}%
				{\string\newlabel{#1}{{\thecexp}{\thepage}}}\thecexp}}
		\catcode`\@=12
		
	\title{Numerical analysis of optimal control problems governed by fourth-order linear elliptic equations using the Hessian discretisation method}
\author{Devika Shylaja \footnote{Department of Mathematics, Indian Institute of Technology Madras, Chennai - 600 036, Tamil Nadu, India. {devus09@gmail.com}}}
		
\begin{document}
			\maketitle

\maketitle
\begin{abstract}
This paper focusses on the optimal control problems governed by fourth-order linear elliptic equations with clamped boundary conditions in the framework of the Hessian discretisation method (HDM).
The HDM is an abstract framework that enables the convergence analysis of numerical methods through a quadruplet known as a Hessian discretisation (HD) and three core properties of HD. The HDM covers several numerical schemes such as the conforming finite element methods, the Adini and Morley non-conforming finite element methods (ncFEMs), method based on gradient recovery (GR) operators and the finite volume methods (FVMs). Basic error estimates and superconvergence results are established for the state, adjoint and control variables in the HDM framework. The article concludes with numerical results that illustrates the theoretical convergence rates for the GR method, Adini ncFEM and FVM.
\end{abstract}

\medskip

{
	
	\smallskip
	
}

\section{Introduction}
This paper discusses the numerical approximation of optimal control problems governed by fourth-order linear elliptic equations with clamped boundary conditions using the Hessian discretisation method (HDM). The HDM, an abstract framework that covers several numerical schemes and establishes a unified convergence analysis for fourth-order linear and semi-linear elliptic partial differential equations, is discussed in \cite{HDM_linear}, \cite{DS_HDM} and \cite{JDNNDS_HDM}. 

\medskip
The HDM is an extension to fourth-order equations of the gradient discretisation method \cite{gdm}, developed for linear and non-linear second-order elliptic and parabolic problems. The HDM is based on a quadruplet given by a discrete space and three reconstructed operators known as a Hessian discretisation (HD) and three abstract properties namely, coercivity, consistency and limit-conformity of HD. These properties which are independent of the model ensure the convergence of a HDM. Some examples of numerical schemes that fit into the HDM framework are the conforming finite element methods (FEMs), the Adini and Morley non-conforming finite element methods (ncFEMs), the finite volume methods (FVMs) and a method based on gradient recovery (GR) operators, see \cite{ciarlet1978finite,biharmonicFV,biharmonicP1fem,BL_FEM,BLAM_Approx.proGRO} for a detailed analysis of these methods. A generic error estimate is established in $L^2$, $H^1$ and $H^2$-like norms in the HDM framework in \cite{HDM_linear}. Further, an improved $L^2$ and $H^1$-like error estimates compared to that in the energy norm in the abstract setting are derived in \cite{DS_HDM}. For instance, under regularity assumption on the exact solution, the generic $L^2$ estimate provides a linear rate of convergence for the Morley ncFEM and the GR methods while the improved estimate yields the expected quadratic rate of convergence in $L^2$ norm for these methods.

\medskip

The problems described by fourth-order linear elliptic equations arise from fluid mechanics and solid mechanics such as bending of elastic plates \cite{ciarlet_plate}. In \cite{SRRRWW}, a mixed formulation has been used for the biharmonic control problem on a convex polygonal domain where the state variable is discretized in primal mixed form using continuous piecewise biquadratic finite elements while piecewise constant approximations are used for the control. In \cite{TGNNKP}, a $C^0$ interior penalty method has been proposed and analyzed for distributed optimal control problems governed by the biharmonic operator. An abstract framework for the error analysis of discontinuous finite element methods applied to control constrained optimal control problems is established in \cite{SCTGAKN}. It is well-known that for piecewise constant control discretisations, the order of convergence in $L^2$ norm cannot be better than linear for the control variable. We refer to \cite{CMAR} for a superconvergence result for the control variable which is based on a post processing step. {\it{To the best of our knowledge, the numerical analysis of optimal control problems using the Adini ncFEM, method based on GR operators and the FVMs have not been studied in literature.}}

\medskip

In this paper, the optimal control problems governed by general fourth-order linear elliptic equations are discretised using the Hessian discretisation method. The basic error estimates applied to the control problems are established in a very generic setting with the help of three core properties associated with the HD. As a result, for these problems, all the schemes entering the HDM framework converge, in particular, conforming FEMs, Adini and Morley ncFEMs, FVMs and GR methods. Under regularity assumptions, if the control variable is approximated by piecewise constant functions, then the basic error estimate yields linear convergence rate for the control variable for the aforementioned FEMs and GR methods. However, a quadratic rate of convergence can be achieved using a post-processing step and following the ideas of \cite{GDM_control,CMAR} taking into account of the additional challenges offered by fourth-order problems in the HDM framework. The companion operator enables us to verify a discrete version of Sobolev embedding for ncFEMs and GR methods, which plays a major role in the analysis of superconvergence result. Companion operators are known to exist for these FEMs, but {\it one for the gradient recovery method with specific property (error between the discrete function and the corresponding companion in $L^2$ norm is small enough) is constructed for the first time}. Results of numerical experiments are presented for the Adini ncFEMs, GR method and FVMs. In the numerical implementation, the discretisation problem is solved using the primal-dual active set strategy \cite{tf}. 

\medskip

The paper is organised as follows. Section \ref{sec.wfcontrol} deals with the optimal control problem  governed by fourth-order linear elliptic equation with clamped boundary conditions.  Section \ref{sec.HDMcontrol} introduces the Hessian discretisation method for the optimal control problem. Some examples of HDM are briefly presented in Section \ref{sec.examplesofHDM}. Section \ref{se.propertiesHDM} discusses the three core properties required to prove the convergence analysis. The main results are presented in Section \ref{sec.mainresults.HDMcontrol}. Section \ref{sec.basicerr.HDMcontrol} establishes basic error estimates for the control, state and adjoint variables in the HDM framework. The superconvergence result is presented in Section \ref{sec:superconvergencel.HDMcontrol} with the help of a post-processing step and a few generic assumptions on the Hessian discretisation which are discussed in detail for conforming FEMs, Adini and Morley ncFEMs, and GR methods.  The article concludes with numerical results (Section \ref{sec.exampleHDMcontrol}) that illustrates the theoretical convergence rates for the GR method, the Adini ncFEM and the FVM.  
 
 \smallskip
 
 \textbf{Notations}. Let $\O \subset \R^d (d\ge1)$ be a bounded domain and $n$ denotes the unit outward normal to the boundary $\partial \Omega$ of $\O$. Let $\symd(\R)$ be the set of symmetric matrices and 
 $\hessian$ be the Hessian matrix. The tensor product between two matrices $\xi,\phi \in \symd(\R)$ is denoted by $\xi:\phi=\sum_{i,j=1}^{d}\xi_{ij}\phi_{ij}$. Set $\hessian:\xi=\sum_{i,j=1}^{d}\partial_{ij}\xi_{ij}$, where $\partial_{ij}=\frac{\partial^2}{\partial x_i \partial x_j}$. The transpose of a fourth-order symmetric tensor $P= (p_{ijkl})$ is given by $P^{\tau} = (p_{klij})$. Observe that $P^{\tau}\xi : \phi = \xi : P \phi$. The standard $ L^2$ inner product and norm (applied on) $L^2(\O)$, $L^2(\O)^d$ and $L^2(\O;\symd(\R))$ are denoted by $(\cdot,\cdot)$ and $\|{\cdot}\|$.
 
\section{The optimal control problem}\label{sec.wfcontrol}
Consider the following distributed optimal control problem governed by fourth-order linear elliptic equations:
\begin{subequations}\label{model_control}
	\begin{align}
	&  \min_{u \in \Uad} J(u) \,\,\, \textrm{ subject to } \label{costHDMcontrol}\\
	& {\sum_{i,j,k,l=1}^{d}\partial_{kl}(a_{ijkl}\partial_{ij}y) = f +\mathcal Cu  \quad \mbox{ in } \Omega, }\label{state}\\
	&\qquad \qquad \quad y=\frac{\partial y}{\partial n}= 0\quad\mbox{ on $\partial\O$}, \label{bc}
	\end{align}
\end{subequations}
where $u$ is the control variable, $y=y(u)$ is the state variable associated with $u$ and the coefficients $a_{ijkl}$ are measurable bounded functions which satisfy the condition, $ a_{ijkl}=a_{jikl}=a_{ijlk} =a_{klij}$ for $i,j,k,l=1,\cdots,d$. The load function $f$ belongs to $L^{2}(\Omega)$, ${\mathcal C} \in {\mathcal L}(L^2(\omega), L^2(\Omega))$ is a localization operator defined by ${\mathcal C}u(\x) = u(\x) \chi_{\omega}(\x)$, where $\chi_{\omega}$ is the characteristic function of $\omega \subset \Omega$. Here
\begin{equation}
J(u) :=\frac{1}{2}\|y-\yd\|^2  + \frac{\alpha}{2} \|u-\ud \|_{L^2(\omega)}^2\nonumber
\end{equation}
represents the cost functional, $\alpha>0$ is a fixed regularization parameter, $\yd$ is the desired state variable for $y$, $\ud \in L^2(\omega)$ is the desired control variable and $\Uad \subset L^2(\omega)$ is a non-empty, convex and closed admissible space of controls.

\medskip
For a given $u \in \Uad,$ the weak formulation that corresponds to \eqref{state}-\eqref{bc} seeks $y \in H^2_0(\O)$ 
such that 
\be\label{weak_HDMcontrol}
\begin{aligned}
	&a(y,w)=\int_\O (f +\mathcal Cu)w\d\x \qquad \fl  w \in H^2_{0}(\O),
\end{aligned}
\ee
where the continuous bilinear form $a(\cdot,\cdot):H^2_0(\O)\times H^2_0(\O) \to \R$ is defined by $$a(z,w):=\sum_{i,j,k,l=1}^{d}\int_\O a_{ijkl}\partial_{ij}z \partial_{kl}w \d\x = \int_\O \hb z:\hb w \d\x \mbox{ with } \hb w =B\hessian w.$$

Assume that $B$ is constant over $\O$, and satisty the following coercivity property:
\be \label{coer:B.}
\exists \varrho>0\mbox{ such that } \norm{\hb w}{}\ge \varrho\norm{w}{H^2(\Omega)}\quad \fl w\in H^2_0(\Omega).
\ee
Hence, the Lax-Milgram Theorem guarantees both existence and uniqueness of a solution to \eqref{weak_HDMcontrol}.

\smallskip

The Karush-Kuhn-Tracker (KKT) optimality system \cite{jl} for the control problem \eqref{model_control} seeks $(\by,\bp,\bu) \in H^2_0(\O) \times H^2_0(\O) \times \Uad$ such that
\begin{subequations} \label{continuous}
	\begin{align}
& a(\by,w) = (f + {\mathcal C}\bu, w) \,  \qquad \fl w \in H^2_0(\O) \quad \text{(state equation)} \label{state_contcontrol}\\
& a(w,\bp) = (\by-\yd, w)\,   \qquad \fl \: w \in H^2_0(\O) \quad \text{(adjoint equation)} \label{adj_contcontrol}\\
&({\mathcal C}^*\bp+\alpha(\bu-\ud),v-\bu)\geq 0\, \qquad  \fl \: v\in \Uad \quad \text{(optimality condition)} \label{opt_contcontrol}
	\end{align}
\end{subequations}
where $\bp$ is the adjoint state associated with $(\by,\bu)$ and ${\mathcal C}^*$ is the adjoint operator of ${\mathcal C}$. Define the projection operator $P_{[a,b]}:\R\to [a,b]$ by, for all $s\in\R\,,\; P_{[a,b]}(s) := \min\{b, \max\{a, s\}\}.$
From the first order optimality condition \eqref{opt_contcontrol}, the following pointwise relation hold true for a.e. $\x \in \O$ \cite[Theorem 2.28]{tf}:
\be \label{conts.proj}
\bu(\x)=P_{[a,b]}\left(\ud(\x)-\frac{{\mathcal C}}{\alpha}\bp\right).
\ee

\section{The Hessian discretisation method for the control problem}\label{sec.HDMcontrol}

This section starts with the HDM for the elliptic equations and is followed by a few examples of numerical schemes that fit into the HDM framework in Subsection \ref{sec.examplesofHDM}. Three properties, independent of the model, which ensures the convergence of a HDM are discussed briefly in Subsection \ref{se.propertiesHDM}. The HDM for the optimal control problem is presented at the end of this section.


\begin{definition}[$B$-Hessian discretisation \cite{HDM_linear}]\label{HD}
	A $B$-Hessian discretisation for fourth-order linear elliptic equations with clamped boundary conditions is a quadruplet $\disc=(X_{\disc,0},\Pi_\disc,\nabla_\disc,\hbd)$ such that
	\begin{itemize}
		\item $X_{\disc,0}$ is a finite-dimensional space encoding the unknowns of the method,
		\item $\Pi_\disc:X_{\disc,0}  \rightarrow L^2(\O)$ is a linear mapping that reconstructs a function from the unknowns,
		\item $\nabla_\disc:X_{\disc,0}  \rightarrow L^2(\O)^d$ is a linear mapping that reconstructs a gradient from the unknowns,
		\item $\hbd:X_{\disc,0}  \rightarrow L^2(\O;\symd(\R))$ is a linear mapping that reconstructs $\hb=B\hessian$ from the unknowns. It must be chosen such that $\norm{\cdot}{\disc}:=\norm{\hbd \cdot}{}$ is a norm on  $X_{\disc,0}.$
	\end{itemize}
\end{definition}

If $\disc=(X_{\disc,0},\Pi_\disc,\nabla_\disc,\hbd)$ is a $B$-Hessian discretisation, the corresponding scheme, called Hessian scheme, for
\be \label{weak} 
	 {\sum_{i,j,k,l=1}^{d}\partial_{kl}(a_{ijkl}\partial_{ij}\psi) =F \quad \mbox{ in } \Omega, } \quad \psi=\frac{\partial \psi}{\partial n}= 0\quad\mbox{ on $\partial\O$}, 
\ee 
is given by
\be\label{base.Hessian scheme}
\begin{aligned}
	&\mbox{Find $\psi_\disc\in X_{\disc,0}$ such that for any $w_\disc\in X_{\disc,0}$,}\\
	&\int_\O \hbd \psi_\disc:\hbd w_\disc\d\x=\int_\O F \Pi_\disc w_\disc\d\x.
\end{aligned}
\ee
As can be seen here, this Hessian scheme is obtained by replacing the continuous space and operators by their discrete ones in the weak formulation of \eqref{weak}.

\subsection{Examples of HDM}\label{sec.examplesofHDM}
This section deals with some examples of numerical methods that are covered in the HDM approach, see \cite{HDM_linear,DS_HDM} for more details.

Let us first set some notations related to meshes.
\begin{definition}[Polytopal mesh {\cite[Definition 7.2]{gdm}}]\label{def:polymesh}~
	Let $\Omega$ be a bounded polytopal open subset of $\R^d$ ($d\ge 1$). A polytopal mesh of $\O$ is $\polyd = (\mesh,\edges,\centers)$, where:
	\begin{enumerate}
		\item $\mesh$ is a finite family of non empty connected polytopal open disjoint subsets of $\O$ (the cells) such that $\overline{\O}= \dsp{\cup_{\cell \in \mesh} \overline{\cell}}$. For any $\cell\in\mesh$, $|\cell|>0$ is the measure of $\cell$, $h_\cell$ denotes the diameter of $\cell$, $\overline{\x}_\cell$ is the center of mass of $K$, and $n_K$ is the outer unit normal to $K$.
		
		\item $\edges$ is a finite family of disjoint subsets of $\overline{\O}$ (the edges of the mesh in 2D, the faces in 3D), such that any $\edge\in\edges$ is a non empty open subset of a hyperplane of $\R^d$ and $\edge\subset \overline{\O}$. Assume that for all $\cell \in \mesh$ there exists  a subset $\edgescv$ of $\edges$ such that the boundary of $\cell$ is ${\bigcup_{\edge \in \edgescv}} \overline{\edge}$. We then set $\mesh_\edge = \{\cell\in\mesh\,;\,\edge\in\edgescv\}$ and assume that, for all $\edge\in\edges$, $\mesh_\edge$ has exactly one element and $\edge\subset\partial\O$, or $\mesh_\edge$ has two elements and $\edge\subset\O$. Let $\edgesint$ be the set of all interior faces, i.e. $\edge\in\edges$ such that $\edge\subset \O$, and $\edgesext$ the set of boundary faces, i.e. $\edge\in\edges$ such that $\edge\subset \dr\O$. The $(d-1)$-dimensional measure of $\edge\in\edges$ is $|\edge|$, and its centre of mass is $\centeredge$.
		
		\item $\centers = (\x_\cell)_{\cell \in \mesh}$ is a family of points of $\O$ indexed by $\mesh$ and such that, for all  $\cell\in\mesh$,  $\x_\cell\in \cell$.
		Assume that any cell $\cell\in\mesh$ is strictly $\x_\cell$-star-shaped, meaning that 
		if $\x\in\overline{\cell}$ then the line segment $[\x_\cell,\x)$ is included in $\cell$. 
		
	\end{enumerate}
	The diameter of such a polytopal mesh is $h=\max_{K\in\mesh}h_K$. The set of internal vertices of $\mesh$ (resp. vertices on the boundary) is denoted by $\mathcal{V}_{\rm int}$ (resp. $\mathcal{V}_{\rm ext}$). 
\end{definition}

 For $K\in\mesh$, let $\rho_K=\max\{r>0\,:\,B(\overline{\x}_K,r)\subset K\}$
be the maximal radius of balls centred at $\overline{\x}_K$ and included in $K$.
Assume that the mesh is regular, in the sense that there exists $\eta>0$ such that
\be\label{reg:mesh}
{\mbox{ for all }} K\in\mesh\,,\; \eta\ge \frac{{\rm diam}(K)}{\rho_K}.
\ee
Let the space of polynomials of degree at most $k$ in $K$ be denoted by $\mathbb{P}_k(K)$.

\subsubsection{Conforming finite element methods}\label{sec.conforming} The $B$-Hessian discretisation $\disc=(X_{\disc,0},\Pi_\disc,\nabla_\disc,\hbd)$ for conforming FEMs is defined by: $X_{\disc,0}$ is a finite dimensional subspace of $H^2_0(\O)$ and, for $v_\disc\in X_{\disc,0}$, $\Pi_\disc v_\disc=v_\disc$, $\nabla_\disc v_\disc=\nabla v_\disc$ and $\hbd v_\disc=\hb v_\disc$. Examples of conforming finite elements include the Argyris and Bogner--Fox--Schmit (BFS) finite elements, see \cite{ciarlet1978finite} for details.

\subsubsection{Non-conforming finite element methods}\label{sec.ncFEM.eg}
$\bullet$ \textsc{the Adini element \cite{ciarlet1978finite}: }Assume that $\O\subset\R^2$ can be covered by a mesh $\mesh$ made up of rectangles. Figure \ref{ncfem.fig} (left) represents an Adini rectangle $\cell \in \mesh$ with vertices $a_1,\,a_2,\,a_3$ and $a_4$ respectively. Each $v_\disc\in X_{\disc,0}$  is a vector of three values at each vertex of the mesh (with zero values at boundary vertices), corresponding to function and gradient values, $\Pi_\disc v_\disc$ is the function such that the values of $(\Pi_\disc v_\disc)_{|K}\in \mathbb{P}_3(K) \oplus \{x_1x^3_2\} \oplus\{x^3_1x_2\}$ and its gradients at the vertices are dictated by $v_\disc$, $\nabla_\disc v_\disc=\nabla(\Pi_\disc v_\disc)$ and $\hbd v_\disc= \hb_{\mesh}(\Pi_\disc v_\disc)$ is the broken $\hb$ of $\Pi_\disc v_\disc$. 
\begin{figure}
	\begin{center}
		\begin{minipage}[b]{0.4\linewidth}
			{\includegraphics[width=5.cm]{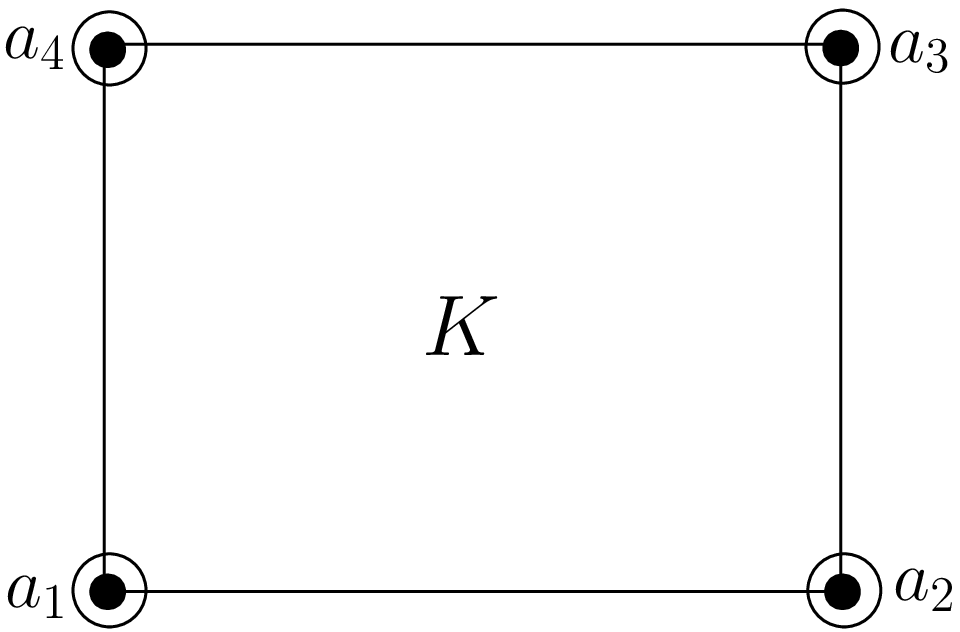}}
		\end{minipage}
		\qquad \quad
		\begin{minipage}[b]{0.35\linewidth}
			{\includegraphics[width=4.3cm]{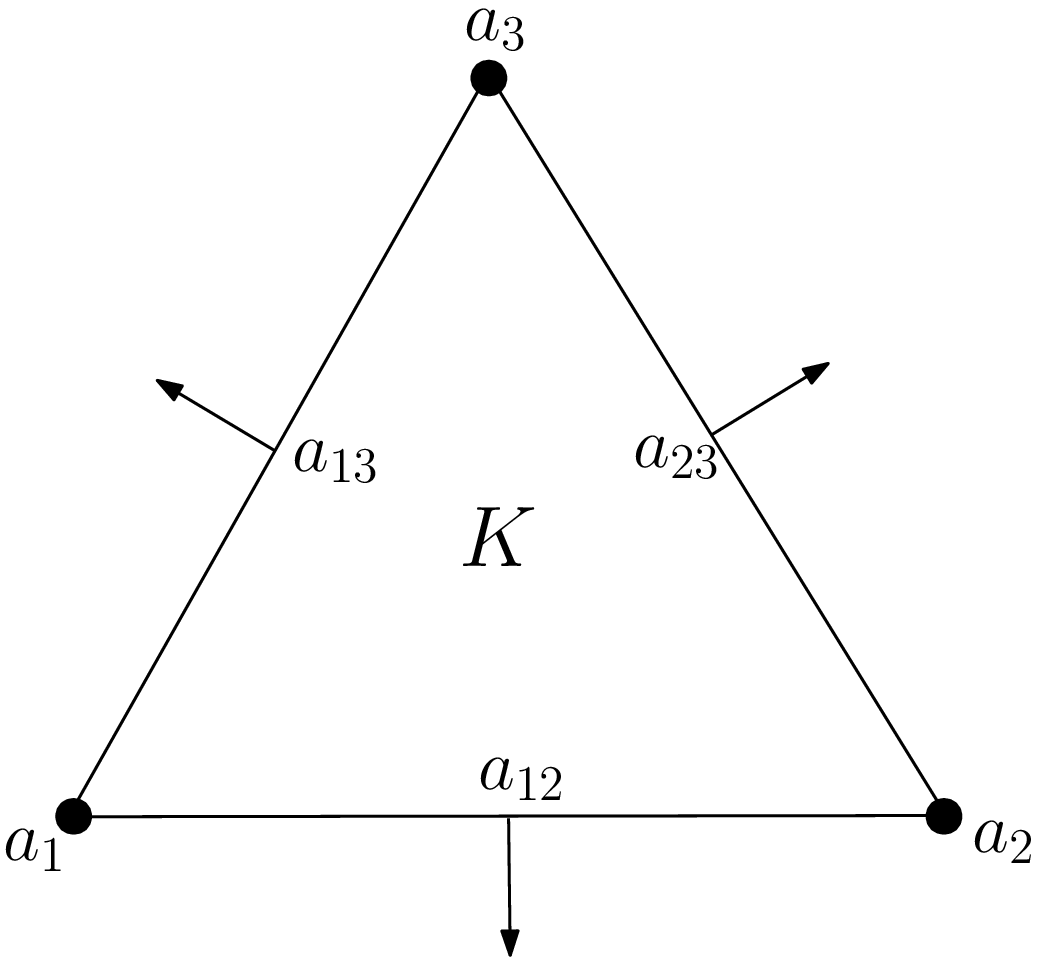}}
		\end{minipage}
		\caption{Adini element (left) and Morley element (right) }\label{ncfem.fig}
	\end{center}
\end{figure}
\smallskip

$\bullet$ \textsc{the Morley element \cite{ciarlet1978finite}: }Let $\mesh$ be a regular conforming triangulation of $\overline{\O}$ into closed triangles (see Figure \ref{ncfem.fig}, right). Each $v_\disc\in X_{\disc,0}$  is a vector of degrees of freedom at the vertices of the mesh (with zero values at boundary vertices) and at the midpoint of the edges opposite to these vertices (with zero values at midpoint of the boundary edges). $\Pi_\disc v_\disc$ is the function such that $(\Pi_\disc v_\disc)_{|K}\in \mathbb{P}_2(K)$ (resp. its normal derivatives) takes the values at the vertices (resp. at the edge midpoints) dictated by $v_\disc$, $\nabla_\disc v_\disc= \nabla_{\mesh}(\Pi_\disc v_\disc)$ is the broken gradient of $\Pi_\disc v_\disc$ and $\hbd v_\disc= \hb_{\mesh}(\Pi_\disc v_\disc)$ is the broken $\hb$ of $\Pi_\disc v_\disc$. 

\subsubsection{Method based on Gradient Recovery Operators}\label{sec.gr}
This method is based on the conforming $\mathbb{P}_1$ finite element space $V_h$
and a gradient recovery operator $Q_h: L^2(\O) \rightarrow V_h$ (see, e.g., \cite[Section 4.2]{HDM_linear} for a GR operator based on biorthogonal systems). This GR operator helps us to lift the non-differentiable piecewise-constant gradient of $\mathbb{P}_1$ finite element functions into the $\mathbb{P}_1$ finite element space itself; the lifted functions are thus differentiable, and can be used to construct a sort of discrete reconstructed Hessian. In order to ensure the coercivity property of this reconstructed Hessian, we consider a stabilisation function $\stab_h\in L^\infty(\O)^d$ with specific design properties \cite[Section 4]{HDM_linear}. Then the $B$--Hessian discretisation based on a triplet $(V_h,Q_h,\stab_h)$ is defined by: $X_{\disc,0}=V_h$ and, for $u_\disc \in X_{\disc,0}$, $\Pi_\disc u_\disc=u_\disc,\,\nabla_\disc u_\disc = Q_h\nabla u_\disc$ and $\hbd u_\disc=B\left[\nabla (Q_h \nabla u_\disc)+\stab_h\otimes (Q_h\nabla u_\disc-\nabla u_\disc)\right],$
where the tensor product $a\otimes b$ of two vectors $a,b\in\R^d$ is the 2-tensor with coefficients $\sum_{i,j=1}^{d}a_ib_j$.
\subsubsection{Finite volume method based on $\Delta$-adapted discretisations}\label{FVM} This section considers the FVM \cite{biharmonicFV} for the biharmonic problem with $\hb=\Delta$ where the mesh satisfy an orthogonality condition known as $\Delta$-adapted mesh (see Figure \ref{fig:diagram} for the two-dimensional case). For all  $\edge \in \edgesint$ with $\mesh_{\edge}= \{\cell, L\}$, the straight line $(\x_\cell, \x_L)$ intersects and is orthogonal to $\edge$, and for all $\edge \in \edgesext$ with $\mesh_{\edge}= \{\cell\}$, the line orthogonal to $\edge$ going through $\x_\cell$ intersects $\edge$. Since $\hb=\Delta$ in this method, one possible choice of $B$ is therefore to set $B\xi=\frac{\rm{tr}(\xi)}{\sqrt{d}}{\rm Id}$ for $\xi \in \symd(\R)$ where ${\rm Id}$ is the identity matrix. This method requires only one unknown per cell and is therefore easy to implement and computational cheap. 

\begin{figure}[h]
	\centering
	\includegraphics[width=0.7\linewidth]{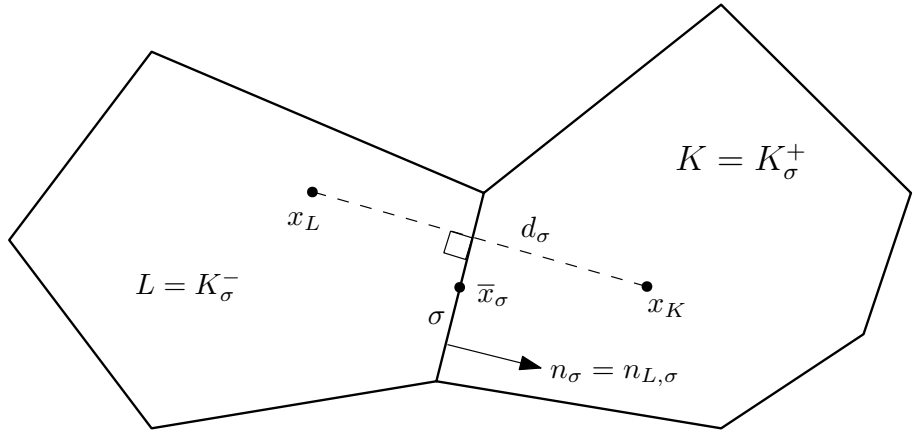}
	\caption{Notations for $\Delta$-adapted discretisation}
	\label{fig:diagram}
\end{figure}

$X_{\disc,0}$ is the space of all real families $v_\disc=(v_\cell)_{\cell \in \mesh}$ such that $v_\cell=0$ if $\cell$ touches $\partial \O$. The operator $\Pi_\disc$ reconstructs a piecewise constant function given by: for any cell $K$, $\Pi_\disc v_\disc = v_K$ on $K$. For $\cell \in \mesh$ and $\edge \in \edgescv$, let $n_{\cell,\edge}$ be the unit vector normal to $\edge$ outward to $\cell$. For all $\edge \in \edges$, we choose an orientation (that is, a cell $K$ such that $\edge\in\edges_K$) and set $n_\edge=n_{\cell,\edge}$. For each $\edge \in \edgesint$, denote by $\cell^-_{\edge}$ and $\cell^+_{\edge}$ the two adjacent control volumes such that the unit normal vector $n_{\edge}$ is oriented from $\cell^-_{\edge}$ to $\cell^+_{\edge}$. For all $\edge \in \edgesext$, denote the control volume $\cell \in \mesh$ such that $\edge \in \edgescv$ by $\cell_{\edge}$ and define $n_{\edge}$ by $n_{\cell,\edge}$. Let
\begin{equation*}
d_\edge=\left\{
\begin{array}{ll}
\dist(\x_{\cell_\edge^-},\edge) + \dist(\x_{\cell_\edge^+},\edge) & {\mbox{ for all }} \edge \in \edgesint \\
\dist(\x_{\cell_\edge},\edge)&  {\mbox{ for all }} \edge \in \edgesext
\end{array}\right.
\end{equation*}
where $\dist(\x_{\cell},\edge)$ denotes the orthogonal distance between $\x_\cell$ and $\edge$. The discrete gradient $\nabla_\disc$ and the Laplace operator $\Delta_\disc$  are defined by their constant values on the cells. 
\begin{equation*}
\nabla_\cell v_\disc=\frac{1}{|\cell|}\sum_{\edge \in \edges_\cell}^{}\frac{|\edge|(\delta_{\cell,\edge} v_\disc)(\centeredge-\x_\cell)}{d_\edge},  \quad  \Delta_\cell v_\disc=\frac{1}{|\cell|}\sum_{\edge \in \edges_\cell}^{}\frac{|\edge|\delta_{\cell,\edge} v_\disc}{d_\edge}, \label{biharmonic_nabla}
\end{equation*}
and set $\hbd v_\disc = \frac{\Delta_\disc v_\disc}{\sqrt{d}}{\rm Id}$, where \begin{equation*}
\delta_{\cell,\edge} v_\disc=\left\{\begin{array}{ll}
v_L-v_\cell &{\mbox{ for all }} \,\edge \in  \edgescv \cap \edgesint\,,\; \mesh_{\edge}= \{\cell, L\} \\
0 & {\mbox{ for all }} \, \edge \in \edgescv \cap \edgesext.
\end{array}\right.
\label{jump}
\end{equation*}

\subsection{Properties of HDM}\label{se.propertiesHDM}
The accuracy of the Hessian scheme can be evaluated using only three measures, all intrinsic to the Hessian discretisation. The first one is a discrete Poincar\'{e} constant, $C_\disc^B$, which controls the norm of the reconstructed function $\Pi_\disc$ and reconstructed gradient $\nabla_\disc$.
\be\label{def.CD}
C_\disc^B = \max_{w_\disc\in X_{\disc,0}\backslash\{0\}} \left(\frac{\norm{\Pi_\disc w_\disc}{}}{\norm{\hbd w_\disc}{}},
\frac{\norm{\nabla_\disc w_\disc}{}}{\norm{\hbd w_\disc}{}}\right).
\ee
The second measure of accuracy  involves an estimate of the interpolation error called consistency in the framework of the HDM and is defined by
\be\label{def.SD}
\begin{aligned}
	&{\mbox{ for all }}  \,\varphi\in H^2_0(\O)\,,\\
	&S_\disc^B(\varphi)=\min_{w_\disc\in X_{\disc,0}}\Big(\norm{\Pi_\disc w_\disc-\varphi}{}
	+\norm{\nabla_\disc w_\disc-\nabla\varphi}{}+\norm{\hbd w_\disc-\hb \varphi}{}\Big).
\end{aligned}
\ee
Finally, the third quantity is a measure of limit-conformity of the HD which measures a defect of the discrete integration-by-parts formula between $\Pi_\disc$ and $\hbd$:
\be\label{def.WD}
\begin{aligned}
	&{\mbox{ for all }} \, \xi \in \wdspace(\O):=\{\zeta\in L^2(\O)^{d \times d}\,;\,\hessian:B^{\tau}B\zeta \in L^2(\O) \}\,,\\
	&W_\disc^B(\xi)=\max_{w_\disc\in X_{\disc,0}\backslash\{0\}}
	\frac{1}{\norm{\hbd w_\disc}{}}\Bigg|\int_\O \Big((\hessian:B^{\tau}B\xi)\Pi_\disc w_\disc- B\xi:\hbd w_\disc \Big)\d\x \Bigg|.
\end{aligned}
\ee
For $\phi\in H^2_0(\O)$, integration-by-parts twice show $\int_{\O}^{}(\hessian:B^{\tau}B\xi)\phi = \int_{\O}^{}B\xi:\hb \phi$. Hence, the quantity in the right-hand side of \eqref{def.WD} measures the error in the discrete Stokes formula. 


\begin{definition}[Coercivity, consistency and limit-conformity]
	Let $(\disc_m)_{m \in \N}$ be a sequence of $B$-Hessian discretisations in the sense of Definition \ref{HD}. We say that
	\begin{enumerate}
		\item $(\disc_m)_{m\in\N}$ is \emph{coercive} if there exists $C_P \in \R^+$ such that $C_{\disc_{m}}^B \leq C_P$ for all $m \in \N$.
		\item $(\disc_m)_{m\in\N}$ is \emph{consistent}, if 
	$$	{\mbox{ for all }}\, \varphi\in H^2_0(\O)\,,
		\lim_{m \rightarrow \infty} S_{\disc_m}^B(\varphi)=0.$$
		\item $(\disc_m)_{m\in\N}$ is \emph{limit-conforming}, if 
	$$	{\mbox{ for all }}\, \xi\in \wdspace(\O)\,,
		\lim_{m \rightarrow \infty} W^B_{\disc_m}(\xi)=0.$$
	\end{enumerate}
\end{definition}

%
%

The error estimates for the HDM applied to the fourth-order linear elliptic equations is stated below. 

\begin{theorem}[Error estimate for Hessian schemes {\cite[Theorem 3.6]{HDM_linear}}]\label{th:error.est.PDE} Under Assumption \eqref{coer:B.}, let $\psi$ be the weak solution to \eqref{weak}. Let $\disc$ be a $B$-Hessian discretisation and $\psi_\disc$ the solution to the corresponding Hessian scheme \eqref{base.Hessian scheme}. Then
	\begin{align*}
	\norm{\Pi_\disc \psi_\disc -\psi}{}& \le C_\disc^B W_\disc^B(\hessian \psi) +(C_\disc^B+1) S_\disc^B(\psi),\\
	\norm{\nabla_\disc \psi_\disc -\nabla\psi}{}
	& \le C_\disc^B W_\disc^B(\hessian \psi) +(C_\disc^B+1) S_\disc^B(\psi),\\
	\norm{\hbd \psi_\disc-\hb \psi}{}
	& \le W_\disc^B(\hessian \psi) +2S_\disc^B(\psi). 
	\end{align*}	
\end{theorem}

An immediate consequence of the error estimate is the following convergence result.
\begin{corollary}[Convergence]
	If $(\disc_m)_{m \in \N}$ is a sequence of $B$-Hessian discretisation	which is coercive, consistent and limit-conforming, then, as $m \rightarrow \infty$, $\Pi_{\dm} \psi_{\dm} \rightarrow \psi$ in $L^2(\O)$, $\nabla_{\dm}\psi_{\dm} \rightarrow \nabla\psi$ in $L^2(\O)^d$ and $\hb_{\dm} \psi_{\dm} \rightarrow \hb \psi$ in $L^2(\O)^{d \times d}$.
\end{corollary}

\begin{remark}[Rates of convergence \cite{HDM_linear, DS_HDM}]\label{rates.PDE_HDM} Under regularity assumption $\psi \in H^4(\O)\cap H^2_0(\O)$, for low--order conforming FEMs, Adini and Morley ncFEMs and gradient recovery methods based on meshes with mesh parameter ``$h$'', $\mathcal O(h)$ estimates can be obtained for $W_\disc^B(\hessian \psi)$ and $S_\disc^B(\psi)$. Theorem \ref{th:error.est.PDE} then gives a linear rate of convergence for these methods. For FVM based on $\Delta$-adapted discretisations, Theorem \ref{th:error.est.PDE} provides an $\mathcal{O}(h^{1/4}|\ln(h)|)$ (in $d=2$) or $\mathcal O(h^{3/13})$ (in $d=3$) error estimate for the Hessian scheme based on the Hessian discretisation.
\end{remark}

\medskip

Let $\disc=(X_{\disc,0},\Pi_\disc,\nabla_\disc,\hbd)$ be a $B$-Hessian discretisation for fourth-order linear equations in the sense of Definition \ref{HD}. Let $\Uadh = \Uad \cap \Uh$, where $\Uh$ is a finite-dimensional subspace of $L^2(\omega)$. Given a $B$-Hessian discretisation $\disc$, the corresponding Hessian scheme for \eqref{continuous} seeks $(\by_\disc,\bp_\disc,\bu_h) \in X_{\disc,0} \times X_{\disc,0} \times \Uadh$ satisfying the KKT optimality conditions \cite{jl}:
\begin{subequations} \label{discrete_kkt.HDMcontrol}
	\begin{align}
	& a_{\disc}(\by_{\disc},w_{\disc}) = (f + \mathcal{C}\bu_{h}, \Pi_\disc w_{\disc}) \,
	&{\mbox{ for all }} \: w_{\disc} \in  X_{\disc,0},  \label{discrete_state.HDMcontrol} \\
	& a_{\disc}(w_{\disc},\bp_{\disc}) = (\Pi_\disc \by_{\disc}-\yd, \Pi_\disc w_{\disc})  \,&  {\mbox{ for all }} \: w_{\disc} \in  X_{\disc,0}, \label{discrete_adjoint.HDMcontrol} \\
	&(\mathcal{C}\Pi_\disc\bp_{\disc} +\alpha(\bu_h-\ud),v_h-\bu_h)\geq     0\,
	&{\mbox{ for all }} \: v_h \in  \Uadh, \label{opt_discrete.HDMcontrol}
	\end{align}
\end{subequations} 
where $$a_{\disc}(\by_{\disc},w_{\disc})= \int_\O \hbd \by_\disc:\hbd w_\disc \d\x.$$
As in the continuous case, existence and uniqueness of a solution to \eqref{discrete_kkt.HDMcontrol} follows from standard variational theorems \cite{tf,jl}. 
 \section{Basic error estimate and super-convergence} \label{sec.mainresults.HDMcontrol}
This section is devoted to the basic convergence results and super-convergence results for the HDM applied to the control problem. The basic error estimate provides a linear rate of convergence on the control problem for the FEMs and the methods based on GR operator under regularity assumptions. Since control is approximated by piecewise constant functions, this is optimal. However, the superconvergence result can be obtained under a superconvergence assumption on the elliptic equations and a few additional assumptions. The proofs of the results stated in this section follow by adapting the corresponding proofs in the gradient discretisation method (GDM)\cite{GDM_control} to accounts for the Hessian discretisation method.

 \subsection{Basic error estimate for the control problem}\label{sec.basicerr.HDMcontrol}
This section establishes the basic error estimate for the control problem. The proof utilizes the stability result of the Hessian schemes that can be obtained through the coercivity of the discrete operators.

 	 \begin{proposition}[Stability of Hessian schemes]\label{prop.stab.HDMcontrol}
 	 	 Let $\disc=(X_{\disc,0},\Pi_\disc,\nabla_\disc,\hbd)$ be a $B$-Hessian discretisation in the sense of the Definiton \ref{HD}. If $\psi_\disc$ is the solution to the Hessian scheme \eqref{base.Hessian scheme}, then
 	 	\be\label{stab.hessian}
 	 	\norm{\hbd \psi_\disc}{}
 	 	\le C_\disc^B\norm{F}{}, \quad \norm{\nabla_\disc \psi_\disc}{}\le 
 	 	(C_\disc^B)^2\norm{F}{}\quad\mbox{and}\quad\norm{\Pi_\disc \psi_\disc}{}\le 
 	 	(C_\disc^B)^2\norm{F}{}.
 	 	\ee
 	 \end{proposition}
 	 
 	 \begin{proof}
 	 	A choice of $w_\disc=\psi_\disc$ in \eqref{base.Hessian scheme}, a Cauchy inequality and the definition of $C_\disc^B$ given by \eqref{def.CD} lead to
 	 	\[
 	 	\norm{\hbd \psi_\disc}{}^2
 	 	\le \norm{F}{}\norm{\Pi_\disc \psi_\disc}{}
 	 	\le C_\disc^B \norm{F}{}\norm{\hbd \psi_\disc}{}.
 	 	\]
 	 	This provides the first inequality in \eqref{stab.hessian}. The remaining two estimates follow from \eqref{def.CD} and the first inequality in \eqref{stab.hessian}. \end{proof}

Recall $\varrho$ from \eqref{coer:B.} and the definition of $C_\disc^B$ from \eqref{def.CD}. The notation $X\lesssim Y$ abbreviates $X\le CY$ for some positive generic constant $C$, which depends only on $\O$, $\varrho$ and an upper bound of $C_\disc^B$.
 
 \smallskip
 
Let $\Prh:L^2(\O) \rightarrow \Uh$  denotes the $L^2$ orthogonal projector on $\Uh$ for the standard scalar product. Recall $\wdspace(\O)$ from \eqref{def.WD}. For $\phi \in H^2(\O)$ with $\hessian \phi \in \wdspace(\O)$,
$$\WS_\disc^B(\phi):= W_\disc^B(\hessian \phi)+S_\disc^B(\phi).$$
 \begin{theorem}[Control estimate]\label{theorem.control.DB.HDMcontrol}
 	Let $\disc$ be a $B$-Hessian discretisation in the sense of Definition \ref{HD}, $(\by,\bp,\bu)$ be the solution to \eqref{continuous} and
 	$(\by_\disc,\bp_\disc,\bu_h)$ be the solution to \eqref{discrete_kkt.HDMcontrol}.
Assume that
 $	\Prh(\Uad)\subset \Uadh.$
 	Then
 	\be\label{est.basic.u.HDMcontrol}
 	\begin{aligned}
 		\sqrt{\alpha}\norm{\bu-\bu_h}{}\lesssim{}& \sqrt{\alpha}\norm{\alpha^{-1}\bp-\Prh(\alpha^{-1}\bp)}{}+(\sqrt{\alpha}+1)\norm{\bu-\Prh\bu}{}\\
 		&+\sqrt{\alpha}\norm{\ud-\Prh\ud}{}+\frac{1}{\sqrt{\alpha}} \WS_\disc^B(\bp)+\WS_\disc^B(\by).
 	\end{aligned}
 	\ee
 \end{theorem}
 \begin{proof}
 	 The proof follows from the continuous and discrete KKT optimality systems given by \eqref{opt_contcontrol} and \eqref{opt_discrete.HDMcontrol}, Proposition \ref{prop.stab.HDMcontrol} and Theorem \ref{th:error.est.PDE}. The analogous result for the GDM applied to the control problem is proved in \cite[Theorem 3.1]{GDM_control} and is therefore omitted. 
 \end{proof}
\begin{remark}\label{remark.basicerr.hdm}
	In particular, the result in Theorem \ref{theorem.control.DB.HDMcontrol} holds true for any bilinear form $a(\cdot,\cdot)$ defined on a subspace $X$ of $L^2(\O)$ and $a_\disc(\cdot,\cdot)$ defined on a discrete space $X_{\disc,0}$ (along with the continuous and discrete KKT optimality conditions), provided the following holds. 
	$$\norm{\Pi_\disc \psi_\disc}{} \lesssim \norm{F}{} \mbox{ and } \norm{\Pi_\disc \psi_\disc-\psi}{} \lesssim \WS_\disc(\psi),$$
	where $\psi$ is the solution to \eqref{weak} and $\psi_\disc$ is the solution to \eqref{base.Hessian scheme}.
\end{remark}
Remark \ref{remark.basicerr.hdm} shows that the basic error estimates for the control problem using the GDM (for second-order problems) and HDM (for fourth-order problems) do not depend on the order of the partial differential equations. This is one of the major advantages of carrying out the analysis in the generic GDM and HDM framework.

 	 
 \begin{proposition}[State and adjoint error estimates]\label{theorem.state.adj.DB.HDMcontrol}
 	Let $\disc$ be a $B$-Hessian discretisation, $(\by,\bp,\bu)$ be the solution to \eqref{continuous} and
 	$(\by_\disc,\bp_\disc,\bu_h)$ be the solution to \eqref{discrete_kkt.HDMcontrol}.
 	Then
 	\begin{align}
 	\label{est.basic.y.HDMcontrol}
 	\norm{\Pi_\disc \by_\disc-\by}{}+\norm{\nabla_\disc \by_\disc-\nabla\by}{}+\norm{\hbd \by_\disc-\hb \by}{} \lesssim{}& \norm{\bu-\bu_h}{}+ \WS_\disc^B(\by),\\
 	\label{est.basic.p.HDMcontrol}
 	\norm{\Pi_\disc \bp_\disc-\bp}{}+\norm{\nabla_\disc \bp_\disc-\nabla\bp}{}+\norm{\hbd \bp_\disc-\hb \bp}{} \lesssim{}& \norm{\bu-\bu_h}{} + \WS_\disc^B(\by) + \WS_\disc^B(\bp).
 	\end{align}		
 \end{proposition}

 The proof is similar to that of \cite[Proposition 3.1]{GDM_control} and hence is skipped.
 \begin{remark}[Rates of convergence for the control problem]\label{rates.control.HDMcontrol}
 	Recalling Remark  \ref{rates.PDE_HDM}, under sufficient smoothness assumption on $\ud$, if $(\by,\bp,\bu)\in H^4(\O)^2 \times H^1(\O)$ and $\Uh$ is made of piecewise constant functions then \eqref{est.basic.u.HDMcontrol}, \eqref{est.basic.y.HDMcontrol} and \eqref{est.basic.p.HDMcontrol} give linear rates of convergence for low-order conforming FEMs, Adini and Morley ncFEMs and GR methods. Also for the FVM, Theorem \ref{theorem.control.DB.HDMcontrol} and Proposition \ref{theorem.state.adj.DB.HDMcontrol} provide an $\mathcal O(h^{1/4}|\ln(h)|)$ (in dimension $d=2$) or $\mathcal O(h^{3/13})$ (in dimension $d=3$) error estimate.
 \end{remark}
 \subsection{Super-convergence for post-processed controls} \label{sec:superconvergencel.HDMcontrol}
This section presents the super-convergence result for the HDM applied to control problem. A post-processing step of control and a few assumptions \assum{1}-\assum{4} (see below) help to derive an improved error estimate for the control, which is known as the superconvergence result. The ideas of this section are motivated by \cite[Section 3.2]{GDM_control} taking into account of the additional challenges offered by fourth-order problems in the HDM framework. These assumptions cover for example, the conforming FEMs, the Adini and Morley ncFEMs, and the GR methods.

\smallskip

 Let  $\mesh$ be a mesh of $\O$ such that each cell $K\in\mesh$ is star-shaped with respect to its centroid $\overline{\x}_K$. Assume that $\omega$ is a polygonal/polyhedral domain such that $\mesh_{|\omega}$ yields a mesh for $\omega$. The admissible set of controls $\Uad$ is defined by
 \begin{equation}\label{Uad:standard.HDMcontrol}
 \Uad:=\{u\in L^2(\omega) \,:\,a\le u(\x)\le b \mbox{ for almost every }\x \mbox{ in } \omega\},
 \end{equation}
where $a$ and $b$ are appropriate constants. The control variable is discretised by piecewise constant functions on this partition and is given by
 \begin{equation}\label{Uh:standard.HDMcontrol}
 \Uh:=\{v:\O\to \R\,:\,{\mbox{ for all }} K\in\mesh\,,\;v_{|K}\mbox{ is a constant}\}.
 \end{equation}
 
 \medskip

Let the projection operator $\Proj:L^1(\O)\to \Uh$ (orthogonal projection on piecewise constant functions on $\mesh$) be defined by 
 \[
 (\Proj v)_{|K} :=\frac{1}{|K|}\int_K v\d\x,\quad  {\mbox{ for all }} (v,K)\in L^1(\O) \times \mesh.
 \]
Let us impose the following assumptions in order to establish the superconvergence result. These are an extension of the assumptions for GDM to HDM.
 
 \begin{itemize}
 	\item[\assum{1}][\emph{Approximation error}]
 	For $w\in H^2(\O)$, there exists $w_\mesh\in L^2(\O)$ such that:
 	\item[i)] 
 	 If $w \in H^4(\O) \cap H^2_0(\O)$ solves $\sum_{i,j,k,l=1}^{d}\partial_{kl}(a_{ijkl}\partial_{ij}w) =g\in L^2(\O)$, and $w_\disc$ is the solution to the corresponding Hessian scheme, then
 	\be\label{state:scv.HDMcontrol}
 	\norm{\Pi_\disc w_\disc-w_\mesh}{}  \lesssim  h^2\norm{g}{}.
 	\ee
 	\item[ii)] For any $w\in H^2(\O)$, it holds
 	\be\label{prop.M.1.HDMcontrol}
 	\begin{aligned}
 		{\mbox{ for all }} v_\disc\in X_{\disc,0}\,,\;
 		\big|(w-w_\mesh,\Pi_\disc v_\disc)\big|
 		\lesssim   h^2\norm{\Pi_\disc v_\disc}{}\norm{w}{H^2(\O)}
 	\end{aligned}
 	\ee
 	and
 	\be\label{prop.M.2.HDMcontrol}
 	\norm{\Proj(w-w_\mesh)}{}\lesssim   h^2\norm{w}{H^2(\O)}.
 	\ee
 	\item[\assum{2}][\emph{Projection estimate}] The estimate $ \norm{\Pi_\disc v_\disc-\Proj(\Pi_\disc v_\disc)}{}  \lesssim  h\norm{\hbd v_\disc}{}$ holds for any $v_\disc\in X_{\disc,0}$.
 	\item[\assum{3}][\emph{Discrete Sobolev imbedding}] For all $v_\disc\in X_{\disc,0}$, it holds $\norm{\Pi_\disc v_\disc}{L^{\infty}(\O)} \lesssim   \norm{\hbd v_\disc}{}.$
 \end{itemize}
 Let  
 \[
 \mesh_2=\{K\in\mesh\,:\,\mbox{$\bu=a$ a.e. on $K$, or $\bu=b$ a.e. on $K$,
 	or $a<\bu<b$ a.e. on $K$}\}
 \]
 be the set of fully active or fully inactive cells, and
 $\mesh_1=\mesh\setminus\mesh_2$ be the set of cells where $\bu$ takes on the value $a$ (resp.\ $b$) as well as values greater than $a$ (resp.\ lower than $b$).
 For $i=1,2$, define $\O_{i,\mesh}:={\rm int}(\cup_{K\in\mesh_i}\overline{K})$.
 The space $W^{1,\infty}(\mesh_1)$ is the usual broken Sobolev space,
 endowed with its broken norm. The last assumption is:
 \begin{itemize}
 	\item[\assum{4}] $|\O_{1,\mesh}|\lesssim h$ and $\bu_{|\O_{1,\mesh}}\in W^{1,\infty}(\mesh_1)$, where $|\cdot|$ denotes the Lebesgue measure in $\R^d$.
 \end{itemize}

 \medskip
  
The assumptions \assum{1}-\assum{3} are verified next for the conforming FEMs, the Adini and Morley ncFEMs, and the method based on GR operators. A discussion on \assum{4} can be found in \cite[Section 3.2.1]{GDM_control}. Recall the definition of $B$-Hessian discretisation $\disc$ for the aforementioned numerical schemes from Section \ref{sec.examplesofHDM}. 

\subsubsection{Conforming FEMs}
For the conforming FEMs, super-convergence result \eqref{state:scv.HDMcontrol} for elliptic equations usually holds with $w_\mesh=w$ (see \cite[Proposition 3.3]{DS_HDM}). In that case, \eqref{prop.M.1.HDMcontrol} and \eqref{prop.M.2.HDMcontrol} are trivially satisfied. Assumption \assum{2} follows from a simple Taylor expansion and the definition of $C_\disc^B$. The continuous Sobolev embedding $H^2(\O)\hookrightarrow L^\infty(\O)$ establishes \assum{3}.

\subsubsection{Nonconforming FEMs}  
A choice of $w_\mesh=w$ and \cite[Proposition 3.3]{DS_HDM} show the first inequality in \assum{1} for the Adini and Morley ncFEMs. For this $w_\mesh$, the remaining two estimates in \assum{1} are trivial. Since $\nabla_\disc v_\disc$ is the classical broken gradient (i.e. the gradient
of $\Pi_\disc v_\disc$ in each cell), the Taylor expansion and the definition of $C_\disc^B$ lead to \assum{2} for both Adini and Morley ncFEMs. Assumption \assum{3} is verified with the help of a companion operator $E_\disc:X_{\disc,0} \rightarrow H^2_0(\O)$. For all $v_\disc \in X_{\disc,0}$, the companion operator in \cite{Morley_plate,Brenner_ncfem} for Adini  and Morley ncFEMs satisfies
\be \label{eq.ED_MorleyAdini}
\norm{\Pi_\disc v_\disc -E_\disc v_\disc}{} \lesssim h^2 \norm{\hbd v_\disc}{}, \quad \norm{\hessian E_\disc v_\disc}{} \lesssim \norm{\hbd v_\disc}{}
\ee
with the coercivity property \eqref{coer:B.} of $B$ in the last step. Note that the range of $E_\disc$ is made of piecewise polynomial functions. A triangle inequality and the continuous Sobolev embedding $H^2(\O)\hookrightarrow L^\infty(\O)$ reveal
$$\norm{\Pi_\disc v_\disc}{L^\infty(\O)}\le  \norm{\Pi_\disc v_\disc-E_\disc v_\disc}{L^\infty(\O)}+ \norm{E_\disc v_\disc}{L^\infty(\O)}\lesssim\norm{\Pi_\disc v_\disc-E_\disc v_\disc}{L^\infty(\O)}+\norm{\hessian E_\disc v_\disc}{}.$$
Let $K \in \mesh$ be such that $ \norm{\Pi_\disc v_\disc-E_\disc v_\disc}{L^\infty(\O)}= \norm{\Pi_\disc v_\disc-E_\disc v_\disc}{L^\infty(K)}$. The inverse estimate \cite[Lemma 1.50]{DG_DA} provides $\norm{\Pi_\disc v_\disc-E_\disc v_\disc}{L^\infty(K)} \lesssim h_K^{-1}\norm{\Pi_\disc v_\disc-E_\disc v_\disc}{L^2(\cell)}$. This, the above displayed estimate and \eqref{eq.ED_MorleyAdini} imply
$\norm{\Pi_\disc v_\disc}{L^\infty(\O)}\lesssim \norm{\hbd v_\disc}{}.$
Thus, assumption \assum{3} is satisfied by the Adini and Morley ncFEMs.

\subsubsection{Gradient Recovery Method}
  The superconvergence result of \assum{1}-i) is proved in \cite[Proposition 3.3]{DS_HDM} with $w_\mesh=w$. Since $w_\mesh=w$, both the estimates in \assum{1}-ii) hold trivially. The following inequality \cite[(4.4)]{HDM_linear} is useful to establish the assumptions \assum{2} and \assum{3}: for all $v_\disc \in X_{\disc,0}:=V_h$,
  \be \label{seminorm_qh}
  C_B^{-1}\sqrt{2}\norm{\hbd v_\disc}{}\ge  \norm{\nabla (Q_h \nabla v_\disc)}{}+\norm{Q_h\nabla v_\disc-\nabla v_\disc}{}.
  \ee
  where $V_h$ is the $H^1_0$-conforming $\Poly_1$ finite element space and $Q_h:L^2(\O) \rightarrow V_h$.  A Taylor expansion and the triangle inequality read $  \norm{\Pi_\disc v_\disc-\Proj(\Pi_\disc v_\disc)}{}\le{}h \norm{\nabla v_\disc}{}
  \le{}h\norm{\nabla v_\disc-Q_h\nabla v_\disc}{}+h\norm{Q_h\nabla v_\disc}{}$. Since $Q_h\nabla v_\disc \in H^1_0(\O)$, the Poincar\'{e} inequality shows that $\norm{Q_h\nabla v_\disc}{} \le {\rm diam}(\O)\norm{\nabla (Q_h\nabla v_\disc)}{}$. The combination of these estimates and \eqref{seminorm_qh} result in
  \begin{align}
  \norm{\Pi_\disc v_\disc-\Proj(\Pi_\disc v_\disc)}{}
  \le{}& hC_B^{-1}\sqrt{2}\max(1,{\rm diam}(\O))
  \norm{\hbd v_\disc}{}\lesssim{} h\norm{\hbd v_\disc}{}.\nonumber
  \end{align}
  Hence \assum{2} follows for the gradient recovery method. 
  
  \smallskip
  
  To establish \assum{3}, let $X_h$ be the Hsieh-Clough-Toucher (HCT) conforming macro finite element \cite{ciarlet1978finite}. Note that the local degrees of freedom of HCT on triangle $\cell \in \mesh$ are the function values and first partial derivatives at the three vertices of $\cell$ in addition to the normal derivative at the midpoints of the edges of $\cell$. Let the set of vertices of $\mesh$ be denoted by $\mathcal V$ and $\centeredge$ be the midpoint of the edge $\edge$. A companion operator $E_\disc:X_{\disc,0} \rightarrow X_h \subset H^2_0(\O)$ is constructed as follows:
  
   Define $E_\disc v_\disc \in X_h$ by setting the degrees of freedom as follows:
   \begin{align}
  {\mbox{ for all }} p \in \mathcal{V}, \quad  E_\disc v_\disc(p)&=v_\disc(p)\label{ed.hct.function}\\
  {\mbox{ for all }} p \in \mathcal{V}, \quad \nabla E_\disc v_\disc(p)&=Q_h \nabla v_\disc(p)\label{ed.hct.gradient}\\
 {\mbox{ for all }} \edge \in \edges,\quad (\nabla E_\disc v_\disc\cdot n_\edge)(\centeredge)&=(Q_h \nabla v_\disc\cdot n_\edge) (\centeredge).\label{ed.hct.normal}
    \end{align}
 The triangle inequality shows that $\norm{\Pi_\disc v_\disc}{L^\infty(\O)}=\norm{ v_\disc}{L^\infty(\O)}\le  \norm{v_\disc-E_\disc v_\disc}{L^\infty(\O)}+ \norm{E_\disc v_\disc}{L^\infty(\O)}.$ Let $K \in \mesh$ be such that $\norm{v_\disc-E_\disc v_\disc}{L^\infty(\O)}=\norm{v_\disc-E_\disc v_\disc}{L^\infty(K)}$ and an inverse estimate implies $\norm{v_\disc-E_\disc v_\disc}{L^\infty(K)}\lesssim h_\cell^{-1}\norm{v_\disc-E_\disc v_\disc}{L^2(\cell)}$. A combination of these estimates and the continuous Sobolev embedding $H^2(\O)\hookrightarrow L^\infty(\O)$ leads to
 \be \label{gr.ed}
 \norm{\Pi_\disc v_\disc}{L^\infty(\O)} \lesssim h_\cell^{-1}\norm{v_\disc-E_\disc v_\disc}{L^2(\cell)}+ \norm{\hessian E_\disc v_\disc}{}.
 \ee
 Let $w$ be a polynomial function on $\cell$. The scaling argument \cite{Gallistl_Morley} reads
   \begin{align*}
   \norm{w}{L^2(\cell)}^2 &\approx \sum_{N \in \mathcal N(\cell)}^{}(\mbox{diam}(\cell))^{2(1+\mathfrak{O}(N))}(N(w))^2,
   \end{align*}
   where $\mathcal{N}(\cell)$ is the set of degrees of freedom and $\mathfrak{O}(N)$ is the order of differentiation in the degrees of freedom. Here, \eqref{ed.hct.function} is of order 0 and \eqref{ed.hct.gradient} and \eqref{ed.hct.normal} are of order 1. Since $v_\disc -E_\disc v_\disc \in \mathbb{P}_3$ on a submesh and $N(v_\disc -E_\disc v_\disc)=0$ if $N$ is of type \eqref{ed.hct.function},
  \begin{align*}
 \norm{v_\disc -E_\disc v_\disc}{L^2(\cell)}^2 &\approx \sum_{N \in \mathcal N(\cell)}^{}h_\cell^{4}(N(v_\disc -E_\disc v_\disc))^2.
 \end{align*}
 This and the definition of $E_\disc$ imply
\begin{align*}
 h_\cell^{-4}\norm{v_\disc -E_\disc v_\disc}{L^2(\cell)}^2 &\approx \sum_{p \in \mathcal{V}_\cell}^{}|(\nabla v_\disc -Q_h \nabla v_\disc)(p)|^2+
\sum_{\edge \in \edgescv}^{}|((\nabla v_\disc -Q_h \nabla v_\disc)\cdot n_\edge)(\centeredge)|^2,
\end{align*}
	where $\mathcal{V}_\cell$ is the set of vertices associated with $\cell$. The above displayed estimate, an inverse estimate and \eqref{seminorm_qh} lead to
	\begin{align*}
	h_\cell^{-4}\norm{v_\disc -E_\disc v_\disc}{L^2(\cell)}^2 \lesssim \norm{\nabla v_\disc -Q_h \nabla v_\disc}{L^{\infty}(\cell)^2}^2 &\lesssim h_\cell^{-2}\norm{\nabla v_\disc -Q_h \nabla v_\disc}{L^2(\cell)^2}^2 \lesssim h_\cell^{-2}\norm{\hbd v_\disc}{}^2.
	\end{align*}
	Therefore,  $h_\cell^{-1}\norm{v_\disc -E_\disc v_\disc}{L^2(\cell)} \lesssim \norm{\hbd v_\disc}{}.$ This and \eqref{gr.ed} reveal
\begin{align}
\norm{\Pi_\disc v_\disc}{L^\infty(\O)}&\lesssim \norm{\hbd v_\disc}{}+ \norm{\hessian E_\disc v_\disc}{}.\label{gr.ed.linfty}
\end{align}	
 An introduction of $\nabla Q_h\nabla v_\disc$, a triangle inequality and an inverse estimate  \cite[Lemma 1.44]{DG_DA} proves
 \begin{align}
 \norm{\hessian E_\disc v_\disc}{}&\lesssim h^{-1}\norm{\nabla E_\disc v_\disc- Q_h\nabla v_\disc}{}+\norm{\nabla Q_h\nabla v_\disc}{}.\label{HED}
 \end{align}
 Since $\nabla E_\disc v_{\disc}-Q_h \nabla v_{{\disc}_|K} \in \mathbb{P}_2(K)$, the definition of $E_\disc$ shows that the first term on the right-hand side of \eqref{HED} vanishes. This, \eqref{HED} and \eqref{seminorm_qh} result in $\norm{\hessian E_\disc v_\disc}{}\lesssim \norm{\hbd v_\disc}{}.$ A substitution of this in \eqref{gr.ed.linfty} yields
$$\norm{\Pi_\disc v_\disc}{L^\infty(\O)}\lesssim \norm{\hbd v_\disc}{}.$$
 Thus, assumption \assum{3} is verified for the gradient recovery methods.
  
  \medskip
  


Recall the mesh regularity parameter $\eta$ from \eqref{reg:mesh}. The notation $X\lesssim_{\eta} Y$ abbreviates $X\le CY$ for some generic positive constant $C$, where $C$ depends only on $\O$, $B$, $\varrho$, an upper bound of $C_\disc$, and $\eta$.
  \medskip
 
 The post-processed continuous and discrete controls are defined by
 \begin{equation} 
 \begin{aligned}
 \tu(\x)={}&P_{[a,b]}\left(\Proj \ud(\x)-\frac{{\mathcal C}}{\alpha}\bp_\mesh(\x) \right),\quad \tu_{h}(\x)={}P_{[a,b]}\left(\Proj\ud(\x)-\frac{{\mathcal C}}{\alpha} \Pi_\disc \bp_\disc(\x)\right)
 \end{aligned}
 \label{Projection2.HDMcontrol}
 \end{equation}
 where $\bp_\mesh$ is defined as in \assum{1}.
 
 \medskip

 \noindent The super-convergence result for post-processed controls is stated next. The proof is obtained by modifying the proof of \cite[Theorem 3.4]{GDM_control} for GDM to HDM by adapting the assumptions \assum{1}-\assum{4} and is therefore omitted.
 \begin{theorem}[Super-convergence for post-processed controls] \label{thm.superconvergence.HDMcontrol}
 	Let $\disc$ be a $B$-Hessian discretisation and $\mesh$ be a mesh.
 	Assume that 
 	\begin{itemize}
 		\item $\Uad$ and $\Uh$ are given by \eqref{Uad:standard.HDMcontrol} and \eqref{Uh:standard.HDMcontrol},
 		\item \assum{1}--\assum{4} hold,
 		\item $\by$ and $\bp$ belong to $H^4(\O)$, $\ud$ belongs to $H^2(\O)$,
 	\end{itemize}
 	and let $\tu$, $\tu_h$ be the post-processed controls defined by \eqref{Projection2.HDMcontrol}.
 	Then there exists $C$ depending only on $\alpha$ such that
 	\be\label{eq:supercv.HDMcontrol}
 	\norm{\tu-\tu_{h}}{} \lesssim_{\eta}Ch^{2}\left(\norm{\bu}{W^{1,\infty}(\mesh_1)}+\mathcal F(a,b,\yd,\ud,f,\by,\bp)\right),
 	\ee
 	where $\mathcal F(a,b,\yd,\ud,f,\by,\bp)$ is defined by
 	\begin{align*}
 	\mathcal F(a,b,\yd,\ud,f,\by,\bp):={}&\minmod(a,b)+\norm{\yd}{}+  \norm{\ud}{H^2(\O)}+\norm{f}{}+\norm{\by}{H^4(\O)}+\norm{\bp}{H^4(\O)}
 	\end{align*}
 	with $\minmod(a,b)=0$ if $ab\le 0$ and $\minmod(a,b)=\min(|a|,|b|)$ otherwise.
 \end{theorem}

 
 \begin{corollary}[Super-convergence for the state and adjoint variables] \label{cor.superconvergence.HDMcontrol}
 	Let $(\by,\bp)$ and $(\by_\disc,\bp_\disc)$ be the solutions to \eqref{state_contcontrol}--\eqref{adj_contcontrol} and  \eqref{discrete_state.HDMcontrol}--\eqref{discrete_adjoint.HDMcontrol}. Under the assumptions of Theorem \ref{thm.superconvergence.HDMcontrol}, the following error estimates hold, with $C$ depending only on $\alpha$:
 	\begin{align}
 	\label{eq_supercv.y.HDMcontrol}
 	\norm{\by_\mesh-\Pi_\disc \by_\disc}{}\lesssim_{\eta}{}&
 	Ch^{2}\left(\norm{\bu}{W^{1,\infty}(\mesh_1)}+\mathcal F(a,b,\yd,\ud,f,\by,\bp)\right),\\
 	\label{eq_supercv.p.HDMcontrol}
 	\norm{\bp_\mesh-\Pi_\disc \bp_\disc}{}\lesssim_{\eta}{}&
 	Ch^{2}\left(\norm{\bu}{W^{1,\infty}(\mesh_1)}+\mathcal F(a,b,\yd,\ud,f,\by,\bp)\right),
 	\end{align}
 	where $\by_\mesh$ and $\bp_\mesh$ are defined as in \assum{1}.
 \end{corollary}
The proof is similar to that of \cite[Corollary 3.7]{GDM_control} and hence is skipped.

\medskip

 The linear control problem considered in this paper helps us to extend our analysis to control problem governed by non-linear elliptic equations in the HDM framework, which is a plan of future work. 	

	\section{Numerical results}\label{sec.exampleHDMcontrol}
	 This section deals with the results of the numerical experiments for distributed optimal control problem governed by the biharmonic equation to support the theoretical estimates obtained in the previous section with $\omega=\Omega$, that is, $\mathcal{C}=\rm{Id}$. Three specific schemes are used for the state and adjoint variables: Adini ncFEM, GR method and FVM. The control variable is discretised using piecewise constant functions. The discrete solution $(\by_\disc,\bp_\disc,\bu_h)$ is computed by using the primal-dual active set algorithm \cite[Section 2.12.4]{tf}. 
	 
	 \smallskip
	 
	 The model problem is constructed in such a way that the exact solution $(\by,\bp,\bu)$ is known. For a given $\bp$ and $\bu_d$, the continuous control $\bu$ is computed using the projection relation \eqref{conts.proj} and the discrete control $\bu_h$ is defined by
	 $$
	 \bu_{h}(\x):=P_{[a,b]}\left(\Proj\left(\ud(\x) -\frac{1}{\alpha}\Pi_\disc \bp_\disc(x)\right) \right)
	 $$
	 which follows from the discrete optimality condition \eqref{opt_discrete.HDMcontrol}.The source term $f$ and the desired state  $\yd$ are the computed using
	 \begin{equation*}
	 f=\Delta^2 \by-\bu, \quad \yd=\by-\Delta^2 \bp.
	 \end{equation*}
	 Let the relative errors be defined by
	 \[
	 \err_\disc(\by):=\frac{\norm{\Pi_\disc \by_\disc -\by}{}}{\norm{\by}{}},\quad
	 \err_\disc(\nabla\by) :=\frac{\norm{\nabla_\disc \by_\disc -\nabla\by}{}}{\norm{\nabla\by}{}}, \quad \err_\disc(\hessian\by) :=\frac{\norm{\hbd \by_\disc -\hessian\by}{}}{\norm{\hessian\by}{}}
	 \]
	 \[
	 \err_\disc(\bp):=\frac{\norm{\Pi_\disc \bp_\disc -\bp}{}}{\norm{\bp}{}},\quad
	 \err_\disc(\nabla\bp) :=\frac{\norm{\nabla_\disc \bp_\disc -\nabla\bp}{}}{\norm{\nabla\bp}{}}, \quad \err_\disc(\hessian\bp) :=\frac{\norm{\hbd \bp_\disc -\hessian\bp}{}}{\norm{\hessian\bp}{}}
	 \]
	 \[
	 \err(\bu):=\frac{\norm{\bu_h -\bu}{}}{\norm{\bu}{}}
	 \quad\mbox{ and }\quad
	 \err(\tu) :=\frac{\norm{\tu_h - \tu}{}}{\norm{\bu}{}},
	 \]
	where $\tu$ and $\tu_h$ are the continuous and discrete post-processed control from \eqref{Projection2.HDMcontrol}. 

	
	 \subsection{Gradient Recovery Method}
	 Here, $X_{\disc,0}$ is the conforming $\mathbb{P}_1$ finite element space on triangular meshes
	 and the implementation was done following the ideas of \cite{BL_stabilizedMFEM} where the gradient recovery operator is constructed using biorthogonal bases. We have already observed in \cite[Section 6.1]{HDM_linear} that the stabilisation factor has only very small impact on the relative errors for the biharmonic problem and thus, as expected, so does the scheme for the entire control problem. Hence, in this section, the stabilisation factor is chosen to be 1. 
	 
	 \subsubsection{Example 1}\label{example1}
	   This example is taken from \cite[Section 4.2]{SRRRWW}. Let the computational domain  be $\O=\left(0,1\right)^2$. The data in the optimal distributed control problem are chosen as follows:
	 \begin{align*}
	 &\by=\sin^2(\pi x)\sin^2(\pi y), \quad \bp=\sin^2(\pi x)\sin^2(\pi y),\\
	 &   \ud=0, \quad \alpha=10^{-3},\quad\Uad=[-750,-50],\quad \bu=P_{[-750,-50]}\left(-\frac{{1}}{\alpha}\bp\right).
	 \end{align*}
	  Table \ref{table.gr} shows the error estimates and the orders of convergence for the control, the post-processed control, the state and the adjoint variables. Observe that linear orders of convergence are obtained for the state and adjoint variable in the energy norm, and quadratic orders of convergence are obtained for state and adjoint variables in $L^{2}$ and $H^1$ norm. Also, linear order of convergence for the control variable, and quadratic order of convergence for the post-processed control are obtained in $L^{2}$ norm. The theoretical rates of convergence in Theorem \ref{theorem.control.DB.HDMcontrol}, Proposition \ref{theorem.state.adj.DB.HDMcontrol}, Remark \ref{rates.control.HDMcontrol}, Theorem \ref{thm.superconvergence.HDMcontrol} and Corollary \ref{cor.superconvergence.HDMcontrol} are confirmed by these numerical outputs.
	\begin{table}[h!!]
		\caption{\small{(GR) Convergence results for the relative errors, Example 1}}
			{\small{\footnotesize
\begin{center}
	\begin{tabular}{ ||c||c|c||c | c ||c|c|| c|c||c|c||c|c||}
		\hline
$h$ &$\err_\disc(\by)$ & Order  & $\err_\disc(\nabla \by)$ & Order  &$\err_\disc(\hessian \by)$ & Order  &$\err(\bu)$ & Order\\ 
	\hline\\[-10pt]  &&\\[-10pt]
0.353553& 2.192387&    -&   0.692406&  -&   0.817825&     -& 0.537029&-\\
0.176777& 0.131323&  4.0613& 0.079054& 3.1307& 0.245715&  1.7348& 0.190741& 1.4934\\
0.088388& 0.032735&  2.0042& 0.019531& 2.0171& 0.116596& 1.0755&  0.081011&   1.2354\\
0.044194& 0.008220&  1.9936& 0.004757& 2.0376& 0.057374& 1.0230&  0.038235&  1.0832\\
0.022097& 0.002081&  1.9821& 0.001215& 1.9695& 0.028479&  1.0105& 0.018865& 1.0192\\		
	\hline				
\end{tabular}
	\end{center}
\begin{center}
	\begin{tabular}{||c||c|c||c | c ||c|c|| c|c||c|c||c|c||}
		\hline
		$h$ &$\err_\disc(\bp)$ & Order  & $\err_\disc(\nabla \bp)$ & Order  &$\err_\disc(\hessian \bp)$ & Order  &$\err(\tu)$ & Order\\ 
		\hline\\[-10pt]  &&\\[-10pt]
0.353553& 3.132234&      -& 0.721611&     - & 0.855785&  -&     0.593791&    -\\
0.176777& 0.145384& 4.4293& 0.099972& 2.8516& 0.246647& 1.7948& 0.126971& 2.2255\\
0.088388& 0.036226& 2.0048& 0.023097& 2.1138& 0.116471& 1.0825& 0.032031&  1.9870\\
0.044194& 0.009068& 1.9982& 0.005552& 2.0567& 0.057308& 1.0231& 0.007716& 2.0536\\
0.022097& 0.002261& 2.0037& 0.001363& 2.0266& 0.028470& 1.0093& 0.001874& 2.0416\\
		\hline				
	\end{tabular}
\end{center}	
		}}\label{table.gr}
\end{table}	
\subsubsection{Example 2}\label{example2}
 Consider the non-convex L-shaped domain given by $\Omega=(-1,1)^2 \setminus\big{(}[0,1)\times(-1,0]\big{)}$. The source term $f$ and the observation $\by_d$ are chosen such that the model problem has the exact singular solution {{\cite[Section 3.4.1]{grisvard}}} given by 
$$
\by(r,\theta)=\bp(r,\theta)=(r^2 \cos^2\theta-1)^2 (r^2 \sin^2\theta-1)^2 r^{1+ \gamma}g_{\gamma,\omega}(\theta)
$$
where $ \gamma\approx 0.5444837367$ is a non-characteristic 
root of $\sin^2( \gamma\omega) =  \gamma^2\sin^2(\omega)$, $\omega=\frac{3\pi}{2}$, and
$g_{\gamma,\omega}(\theta)=(\frac{1}{\gamma-1}\sin ((\gamma-1)\omega)-\frac{1}{ \gamma+1}\sin(( \gamma+1)\omega))(\cos(( \gamma-1)\theta)-\cos(( \gamma+1)\theta))$ 
$-(\frac{1}{\gamma-1}\sin(( \gamma-1)\theta)-\frac{1}{ \gamma+1}\sin(( \gamma+1)\theta))
(\cos(( \gamma-1)\omega)-\cos(( \gamma+1)\omega)).$ The exact control $\bar{u}$ is chosen as $\bar{u}=P_{[-600,-50]} ( -
\frac{1}{ \alpha}\bp )$, where $\alpha=10^{-3}$. 

\smallskip

The errors and orders of convergence for the numerical approximation to the control problem are presented in Table \ref{table.gr.Lshaped}. This example is particularly interesting since the solution is less regular due to the corner singularity. Since $\O$ is non-convex, we obtain only suboptimal orders of convergence for the state and adjoint variables in the energy, $H^1$ and $L^2$ norms. This can be clearly seen from the table. Also, observe suboptimal order of convergence for the post processed control in $L^2$ norm. However, the control converges at the optimal rate of $h$.

	\begin{table}[h!!]
	\caption{\small{(GR) Convergence results for the relative errors, Example 2}}
	{\small{\footnotesize
			\begin{center}
				\begin{tabular}{ ||c||c|c||c | c ||c|c|| c|c||c|c||c|c||}
					\hline
					$h$ &$\err_\disc(\by)$ & Order  & $\err_\disc(\nabla \by)$ & Order  &$\err_\disc(\hessian \by)$ & Order  &$\err(\bu)$ & Order\\ 
					\hline\\[-10pt]  &&\\[-10pt]

0.353553& 0.291494& -&0.348393& -& 0.680812& -& 0.436002 &-\\
0.176777 &0.105991& 1.4595& 0.124221&1.4878& 0.267124& 1.3497& 0.184329 &1.2421\\
0.088388& 0.033654& 1.6551& 0.042625& 1.5431& 0.125727& 1.0872& 0.070284& 1.3910\\
0.044194& 0.009829& 1.7757& 0.013149& 1.6967& 0.066726& 0.9140& 0.028551& 1.2997\\ 
0.022097& 0.002745& 1.8403& 0.003927& 1.7436& 0.039001& 0.7748& 0.012847& 1.1521\\	
					\hline				
				\end{tabular}
			\end{center}
			\begin{center}
				\begin{tabular}{||c||c|c||c | c ||c|c|| c|c||c|c||c|c||}
					\hline
					$h$ &$\err_\disc(\bp)$ & Order  & $\err_\disc(\nabla \bp)$ & Order  &$\err_\disc(\hessian \bp)$ & Order  &$\err(\tu)$ & Order\\ 
					\hline\\[-10pt]  &&\\[-10pt]
0.353553& 0.447020& -&0.377361& -&0.440887& -& 0.421946& -\\
0.176777& 0.177536& 1.3322& 0.142128& 1.4088& 0.217775& 1.0176& 0.160613& 1.3935\\
0.088388& 0.059354& 1.5807& 0.046061& 1.6256& 0.115940& 0.9095& 0.051943& 1.6286\\
0.044194& 0.018013& 1.7203& 0.013879& 1.7307& 0.064817& 0.8389& 0.015548& 1.7402\\
0.022097& 0.005357& 1.7496& 0.004228& 1.7146& 0.038615& 0.7472& 0.004564& 1.7685\\ 	
					\hline				
				\end{tabular}
			\end{center}	
	}}\label{table.gr.Lshaped}
\end{table}	
\subsection{Adini Nonconforming Finite Element Method}
\subsubsection{Example 1} Consider the example in Section \ref{example1}. Table \ref{table.Adini} presents the numerical results for the control problem. As seen in the table, the rate of convergence is quadratic in $L^2$, $H^1$ and $H^2$-norms for the state and adjoint variables whereas linear for the control variable in $L^2$ norms as expected by the theoretical estimates. Also, the rate of convergence is quadratic for the post-processed control.

	\begin{table}[h!!]
	\caption{\small{(Adini) Convergence results for the relative errors, Example 1}}
	{\small{\footnotesize
			\begin{center}
				\begin{tabular}{ ||c||c|c||c | c ||c|c|| c|c||c|c||c|c||}
					\hline
					$h$ &$\err_\disc(\by)$ & Order  & $\err_\disc(\nabla \by)$ & Order  &$\err_\disc(\hessian \by)$ & Order  &$\err(\bu)$ & Order\\ 
					\hline\\[-10pt]  &&\\[-10pt]
0.707107&0.522904& -& 0.507738& -& 0.498445& -& 0.707839& -\\
0.353553& 0.066834& 2.9679& 0.084387& 2.5890& 0.161496& 1.6259& 0.334984& 1.0793\\
0.176777& 0.031181& 1.0999& 0.032892& 1.3593& 0.045880& 1.8155& 0.180827& 0.8895\\
0.088388& 0.007110& 2.1328& 0.007478& 2.1370& 0.011170& 2.0382& 0.091098& 0.9891\\
0.044194& 0.002055& 1.7907& 0.002133& 1.8100& 0.002903& 1.9438& 0.045888& 0.9893\\
0.022097& 0.000448& 2.1960& 0.000471& 2.1782& 0.000701& 2.0512& 0.022965& 0.9987\\								
					\hline				
				\end{tabular}
			\end{center}
			\begin{center}
				\begin{tabular}{||c||c|c||c | c ||c|c|| c|c||c|c||c|c||}
					\hline
					$h$ &$\err_\disc(\bp)$ & Order  & $\err_\disc(\nabla \bp)$ & Order  &$\err_\disc(\hessian \bp)$ & Order  &$\err(\tu)$ & Order\\ 
					\hline\\[-10pt]  &&\\[-10pt]
0.707107& 0.160991& -& 0.160299& -&0.212052& -& 0.104473&-\\
0.353553& 0.076681& 1.0700& 0.090171& 0.8300& 0.163217& 0.3776& 0.058716& 0.8313\\
0.176777& 0.023139& 1.7285& 0.024586& 1.8748& 0.042476& 1.9420& 0.017007& 1.7876\\
0.088388& 0.006076& 1.9292& 0.006316& 1.9608& 0.010720& 1.9864& 0.004510& 1.9148\\
0.044194& 0.001538& 1.9821& 0.001590& 1.9896& 0.002686& 1.9966& 0.001144& 1.9791\\
0.022097& 0.000386& 1.9958& 0.000398& 1.9976& 0.000672& 1.9992& 0.000287& 1.9928\\
					\hline				
				\end{tabular}
			\end{center}	
	}}\label{table.Adini}
\end{table}
\subsubsection{Example 2} In this example, choose the L-shaped domain and data as in Section \ref{example2}. The errors and orders of convergence for the control, post-processed control, state and adjoint variables are shown in Table \ref{table.Adini.Lshaped}. The same comments as in Section \ref{example2} can be made about the rates of convergence for these variables. 
\begin{table}[h!!]
	\caption{\small{(Adini) Convergence results for the relative errors, Example 2}}
	{\small{\footnotesize
			\begin{center}
				\begin{tabular}{ ||c||c|c||c | c ||c|c|| c|c||c|c||c|c||}
					\hline
					$h$ &$\err_\disc(\by)$ & Order  & $\err_\disc(\nabla \by)$ & Order  &$\err_\disc(\hessian \by)$ & Order  &$\err(\bu)$ & Order\\ 
					\hline\\[-10pt]  &&\\[-10pt]
	0.707107& 0.452221& -& 0.536830& -& 0.593857& -& 0.497598& -\\
	0.353553& 0.071842& 2.6541& 0.086455& 2.6344& 0.155303& 1.9350& 0.234176& 1.0874\\
	0.176777& 0.023752& 1.5968& 0.027716& 1.6412& 0.074929& 1.0515& 0.123440& 0.9238\\
	0.088388& 0.006276& 1.9200& 0.007639& 1.8592& 0.046077& 0.7015& 0.062943& 0.9717\\
	0.044194& 0.001409& 2.1551& 0.002012& 1.9248& 0.030899& 0.5765& 0.031628& 0.9928\\
	0.022097& 0.000462& 1.6102& 0.000705& 1.5129& 0.021104& 0.5500& 0.015847& 0.9970\\				
					
					\hline				
				\end{tabular}
			\end{center}
			\begin{center}
				\begin{tabular}{||c||c|c||c | c ||c|c|| c|c||c|c||c|c||}
					\hline
					$h$ &$\err_\disc(\bp)$ & Order  & $\err_\disc(\nabla \bp)$ & Order  &$\err_\disc(\hessian \bp)$ & Order  &$\err(\tu)$ & Order\\ 
					\hline\\[-10pt]  &&\\[-10pt]
0.707107& 0.088830& -&0.204666&  -& 0.371137& -& 0.057421& -\\
0.353553& 0.024503& 1.8581& 0.046782& 2.1293& 0.137684& 1.4306& 0.016297& 1.8169\\
0.176777& 0.005945& 2.0432& 0.012224& 1.9362& 0.070724& 0.9611& 0.004417& 1.8835\\
0.088388& 0.001470& 2.0162& 0.003713& 1.7189& 0.045454& 0.6378& 0.001376& 1.6830\\
0.044194& 0.000600& 1.2931& 0.001329& 1.4828& 0.030815& 0.5608& 0.000601& 1.1938\\
0.022097& 0.000321& 0.9011& 0.000546& 1.2823& 0.021083& 0.5476& 0.000298& 1.0133\\					
					\hline				
				\end{tabular}
			\end{center}	
	}}\label{table.Adini.Lshaped}
\end{table}	

\subsection{Finite Volume Method}
In this method, the schemes were first tested on a series of regular triangular meshes (\texttt{mesh1} family) and then on square meshes (\texttt{mesh2} family), both taken from \cite{benchmark} for Example \ref{example1}. To ensure that the meshes satisty the orthogonality property, the point $\x_\cell \in \cell$ is chosen as the circumcenter of $K$ if $K$ is a triangle, or the center of mass of $K$ if $K$ is a rectangle. As a result, for triangular meshes, the errors in $L^2$ norm are computed using a skewed midpoint rule, where the circumcenter of each cell is considered instead of its center of mass. Let the relative $H^2$ error be denoted by
\[
\err_\disc(\Delta\by) :=\frac{\norm{\Delta_\disc \by_\disc -\Delta \by}{}}{\norm{\Delta\by}{}}, \quad \err_\disc(\Delta\bp) :=\frac{\norm{\Delta_\disc \bp_\disc -\Delta \bp}{}}{\norm{\Delta\bp}{}}.
\]
The errors of the numerical approximations to state, adjoint and control variables on uniform meshes are shown in Tables \ref{table.fvm.triangle}-\ref{table.fvm.square}.
\begin{table}[h!!]
	\caption{\small{(FV) Convergence results for the relative errors, triangular mesh}}
{\small{\footnotesize
\begin{center}
	\begin{tabular}{ ||c||c|c||c | c ||c|c|| c|c||c|c||c|c||}
			\hline
$h$ &$\err_\disc(\by)$ & Order  & $\err_\disc(\nabla \by)$ & Order  &$\err_\disc(\Delta \by)$ & Order  &$\err(\bu)$ & Order\\ 
		\hline\\[-10pt]  &&\\[-10pt]
0.250000& 0.200670&  -&     0.298405&      -& 0.165136&      -&0.245085&-\\
0.125000& 0.021019& 3.2551& 0.135346& 1.1406& 0.057870& 1.5128&0.116630& 1.0713\\
0.062500& 0.005108& 2.0409& 0.066054& 1.0349& 0.030285& 0.9342&0.057540& 1.0193\\
0.031250& 0.001178& 2.1169& 0.032808& 1.0096& 0.016785& 0.8514&0.028819&  0.9976\\
0.015625& 0.000265& 2.1513& 0.016374& 1.0026& 0.009900& 0.7617&0.014408& 1.0001\\
		\hline				
	\end{tabular}
\end{center}
\begin{center}
\begin{tabular}{ ||c||c|c||c | c ||c|c|| c|c||c|c||c|c||}
						\hline
$h$ &$\err_\disc(\bp)$ & Order  & $\err_\disc(\nabla \bp)$ & Order  &$\err_\disc(\Delta \bp)$ & Order  &$\err(\tu)$ & Order\\ 
	\hline\\[-10pt]  &&\\[-10pt]
0.250000&0.230914& -&0.316994&- & 0.189286&   -&  0.094040& -\\
0.125000&0.032775& 2.8167& 0.136993& 1.2104& 0.061257& 1.6276& 0.024703& 1.9286 \\
0.062500&0.007282& 2.1703& 0.066202& 1.0492& 0.030607& 1.0010& 0.004857&  2.3465\\
0.031250&0.001693& 2.1049& 0.032824& 1.0121& 0.016820& 0.8637& 0.001200&   2.0174\\
0.015625&0.000380& 2.1557& 0.016376& 1.0032& 0.009904& 0.7641& 0.000260&  2.2042\\
						\hline				
					\end{tabular}
				\end{center}	
			}}\label{table.fvm.triangle}
		\end{table}		  
In the case of triangular meshes, slightly better quadratic order of convergence for the state and adjoint variables in $L^2$ norm, linear order of convergence in $H^1$ norm and sublinear in $H^2$ norm are obtained. The control converges at the optimal rate of $h$, whereas the post processed control converges with quadratic rate, which is a superconvergence result.

\smallskip		
		
For the square meshes, we obtain quadratic rate of convergence in $L^2,$ $H^1$ and $ H^2$ norms for the state and adjoint variables. This superconvergence result is not entirely surprising, since rectangular meshes are extremely regular and symmetric. Without post-processing, an $\mathcal O(h)$ convergence rate is obtained on the control variable. The numerical results are better than the theoretical rates stated in Theorem \ref{theorem.control.DB.HDMcontrol}, Proposition \ref{theorem.state.adj.DB.HDMcontrol} and Remark \ref{rates.control.HDMcontrol}. Also, using a post-processing step, an improved error estimate for the post-processeed control variable with quadratic order of convergence is obtained numerically.
	\begin{table}[h!]
		\caption{\small{(FV) Convergence results for the relative errors, square mesh}}
		{\small{\footnotesize
				\begin{center}
					\begin{tabular}{ ||c||c|c||c | c ||c|c|| c|c||c|c||c|c||}
						\hline
						$h$ &$\err_\disc(\by)$ & Order  & $\err_\disc(\nabla \by)$ & Order  &$\err_\disc(\Delta \by)$ & Order  &$\err(\bu)$ & Order\\ 
						\hline\\[-10pt]  &&\\[-10pt]
0.353553& 0.288994&      -& 0.196325&      -& 0.270192&      -& 0.398184& -\\
0.176777& 0.060061& 2.2665& 0.045562& 2.1073& 0.056607& 2.2549& 0.187167&  1.0891\\
0.088388& 0.015072& 1.9946& 0.010322& 2.1420& 0.014538& 1.9612& 0.092209& 1.0213\\
0.044194& 0.003700& 2.0263& 0.002590& 1.9945& 0.003551& 2.0334& 0.045989&  1.0036\\
0.022097& 0.000927& 1.9968& 0.000642& 2.0125& 0.000891& 1.9941& 0.022985& 1.0006\\
						\hline				
					\end{tabular}
				\end{center}
				\begin{center}
					\begin{tabular}{ ||c||c|c||c | c ||c|c|| c|c||c|c||c|c||}
						\hline
						$h$ &$\err_\disc(\bp)$ & Order  & $\err_\disc(\nabla \bp)$ & Order  &$\err_\disc(\Delta \bp)$ & Order  &$\err(\tu)$ & Order\\ 
						\hline\\[-10pt]  &&\\[-10pt]
0.353553& 0.300063&      -& 0.189326&      -& 0.285835&   -&    0.144022&-\\
0.176777& 0.066723& 2.1690& 0.039945& 2.2448& 0.065909& 2.1166&0.035315& 2.0280\\
0.088388& 0.016237& 2.0389& 0.009482& 2.0747& 0.016092& 2.0342&0.009726& 1.8604\\
0.044194& 0.004033& 2.0093& 0.002337& 2.0208& 0.003998& 2.0091&0.002471& 1.9766\\
0.022097& 0.001007& 2.0023& 0.000582& 2.0054& 0.000998& 2.0023&0.000589& 2.0693\\
						\hline				
					\end{tabular}
				\end{center}	
			}}\label{table.fvm.square}
		\end{table}	
%
%
 
 \medskip
 
 \textbf{Acknowledgment: }The author would like to sincerely thank Prof J\'er\^ome Droniou and Prof Neela Nataraj for their fruitful comments. 
\bibliographystyle{abbrv}
\bibliography{HDM_nonlinear_elliptic}

\end{document}